\documentclass[12pt,a4paper]{amsart}
\usepackage{graphicx}
\usepackage{amssymb,amstext,amsmath,amsthm}
\usepackage[utf8]{inputenc}
\usepackage{color}
\usepackage[normalem]{ulem}

\newcommand{\domain}{\mathcal{D}}
\newcommand{\set}[1]{\left\{#1\right\}}
\newcommand{\norm}[2]{\| #1 \|_{#2}}

\newcommand{\lr}[1]{\left( #1\right)}
\newcommand{\yd}{y^\delta}
\newcommand{\xp}{x^\dag}

\newcommand{\xad}{x_{\alpha}^\delta}
\newcommand{\xa}{x_{\alpha}}
\newcommand{\xaux}{x_{\mathrm{aux}}}
\newcommand{\aux}{\alpha_{\mathrm{aux}}}
\newcommand{\ra}{r_{\alpha}}

\newcommand{\adp}{\alpha_{\mathrm{DP}}}

\newcommand{\mpe}{M_{p,E}}
\newcommand{\Jaa}{T_{-a,\alpha}}

\newenvironment{frcseries}{\fontfamily{frc}\selectfont}{}

\newcommand{\bigo}{\mathcal O}

\newtheorem*{thm}{Theorem}
\newtheorem{corollary}{Corollary}
\newtheorem{lemma}{Lemma}
\newtheorem{pro}{Proposition}
\newtheorem{ass}{Assumption}
\newtheorem*{DP}{Discrepancy principle}
\theoremstyle{definition}
\newtheorem*{definition}{Definition}
\theoremstyle{remark}
\newtheorem{remark}{Remark}
\newtheorem{xmpl}{Example}

\title[Oversmoothing for non-linear
problems]{Tikhonov regularization with oversmoothing penalty for non-linear ill-posed
  problems in Hilbert scales}
\author{Bernd Hofmann}
\address{Department of Mathematics, Chemnitz University of Technology,
  09107 Chemnitz,  Germany}
\email{bernd.hofmann@mathematik.tu-chemnitz.de}
\author{Peter Math\'e}
\address{Weierstra{\ss} Institute for Applied Analysis and
  Stochastics, Mohrenstra{\ss}e 39, 10117 Berlin,  Germany}
\email{peter.mathe@wias-berlin.de}

\date{}

\begin{document}
\begin{abstract}
  We study Tikhonov regularization for ill-posed non-linear operator equations
  in Hilbert scales. Our focus is on the interplay between the
  smoothness-promoting properties of the  penal\-ty and the smoothness
  inherent in the solution. The objective is to study the situation
  when the unknown solution fails to have a
  finite penalty value, hence when the penalty is oversmoothing.
By now this case was only studied for linear operator equations in
Hilbert scales. We extend those results to certain classes of non-linear problems.
The main result asserts that, under appropriate assumptions, order optimal reconstruction
is still possible. In an appendix we highlight that the non-linearity
assumption underlying the present analysis is met for specific applications.
\end{abstract}
\maketitle

\section{Introduction}
\label{sec:intro}
Tikhonov regularization is a versatile means of stabilizing linear and
non-linear ill-posed operator equations in Hilbert and Banach
spaces. In either case a stable approximate solution is obtained by
minimizing the Tikhonov functional, which consists of two
summands: a term representing the data misfit and a stabilizing
penalty. There is vast literature on the quality of the obtained approximate
solutions, and we mention the monographs~\cite{EHN96,Scherzetal09,SKHK12,Yagola98}. The analysis
always uses the minimizing property, i.e.,\ that the Tikhonov
functional at the minimizer takes smaller values than at the true
solution.

The choice of the penalty is context-dependent. This has consequences for the obtained
approximations, because the corresponding mimimizers will have a
finite penalty value. Some of the relevant choices are
\emph{total-variation} penalties, often used in image reconstruction (cf., e.g., \cite{Burger13}),
since such penalties promote edges, and $\ell^{1}$-penalties (cf., e.g., \cite{Daub04}) and \cite{BurFleHof13,GHS08,Lorenz08}) or  $L^1$-penalties (cf., e.g., \cite{BrePik13,ClaKun12} and \cite[p.~87-89]{Scherzetal09}),
since they have sparsity promoting properties. 

Actually, a smoothing penalty may be used in order to detect whether a
solution has certain smoothness properties. This was studied in the
context of numerical differentiation in~\cite{Wang02} (see also~\cite{Wan06} for extensions). These authors show that, for a
prescribed choice of the regularization parameter,  the penalty must
blow up whenever the solution fails to have a finite penalty value,
and this can be used to detect singularities.

Within the present study the objective is to handle penalties which impose
smooth reconstructions. One reason to do so is the well-known fact
that the classical version of Tikhonov regularization has limited
qualification, which means that it cannot take into account higher solution smoothness.
Hence its error norm convergence rate as a function of the noise
level $\delta$  cannot exceed $\delta^{2/3}$ or~$\delta^{1/2}$,
depending on whether a priori or a posteriori parameter choice is
used.

In all the above cases we face the problem that the unknown solution
may not have the expected feature implied by the penalty. Images may not have only a few sharp
edges, solutions may not have a sparse representation or may not be
smooth. Then the question arises whether the use of such mis-specified
penalties still allows for good reconstructions yielding errors of optimal order with respect to the noise level.
 Here we focus on \emph{smoothness-promoting} penalties, and we shall
 call the penalty \emph{oversmoothing} if the unknown solution element fails
 to have the expected smoothness, and hence the penalty will not have
 a finite value at the solution.

For {\it linear ill-posed problems in Hilbert scales} F.~Natterer has
shown  in the seminal paper~\cite{Natterer84} that oversmoothing makes
no trouble and corresponding results  on convergence rates have been
established. The analysis of linear ill-posed problems
in Hilbert scales is special, because the minimizer of the Tikhonov
functional is given explicitly and the analysis is based on the
explicit representation.

Here we extend those results to some classes of non-linear
operator equations, which have only solutions with \emph{infinite penalty value}, and hence we cannot apply the
arguments, typically used in the proofs. Our `proof architecture' is
also different from that of \cite{Natterer84}, even in the special
case of a linear forward operator, because it uses the discrepancy
principle for the choice of the regularization parameter.
The present approach may be regarded as  a further step towards regularization theory
for linear or non-linear operator equations with oversmoothing penalty.

Specifically, we consider the approximate solution of an (at least
locally) ill-posed operator equation
\begin{equation}
  \label{eq:opeq}
  F(x) =y
\end{equation}
with an (in general) non-linear forward operator $F: \domain(F) \subseteq X \to Y$ between the infinite dimensional Hilbert
 spaces $X$ and $Y$,
 with domain of definition~$\domain(F)$.  By $\xp$ we denote a solution to \eqref{eq:opeq} for given $y$.

In order to consider a Hilbert scale, and to treat an oversmoothing
penalty,  we introduce an {\it unbounded linear self-adjoint operator} $B\colon \domain(B)
\subset X \to X$, which is strictly positive such that we have for some $m>0$
\begin{equation} \label{eq:m}
\norm{Bx}{X}\geq m \norm{x}{X},\quad \mbox{for all} \quad x\in \domain(B).
\end{equation}

Based on noisy data $\yd \in Y$,  obeying the deterministic noise model
\begin{equation} \label{eq:noise}
\|y-\yd\|_Y \le \delta
\end{equation}
with {\it noise level} $\delta>0$, we use regularized solutions $\xad \in \domain$ for
$$ \domain:=\domain(F) \cap \domain(B),$$
being minimizers of the extremal problem
\begin{equation}
  \label{eq:tikhonov}
  T^\delta_{\alpha}(x) := \norm{F(x) - \yd}{Y}^{2} + \alpha\norm{B(x - \bar x)}{X}^{2} \to \min, \;\; \mbox{subject to} \;\;  x\in\domain,
\end{equation}
for the Tikhonov functional $T^\delta_\alpha$, where $\bar
x\in\domain$ is a given smooth reference element.

We shall focus on the oversmoothing case, i.e.,\ when~$\xp\not\in\domain(B)$ such that~$T^\delta_{\alpha}(\xp)=
\infty$.

As already mentioned, in the linear case, i.e.~for a bounded linear operator $F$,
such approach was studied by Natterer in~\cite{Natterer84}.
In a slightly
different setup oversmoothing was studied in~\cite{MathTaut11}. Again,
optimality was shown for a priori parameter choice and for several a
posteriori choices of the regularization parameter~$\alpha$.

The outline is as follows. We formulate our assumptions and the main
result in Section~\ref{sec:main}. The proof of this result is then
given in Section~\ref{sec:proof-main}. Relevant outcomes
of this study are discussed in Section~\ref{sec:discussion}. The Appendix provides us with
several assertions and examples in order
to validate the required non-linearity assumption in specific situations.

\section{Assumptions and main result}
\label{sec:main}

We shall use the following standing assumption in order to guarantee
the required properties of $F,\;\domain(F)$ and $\xp$.

\begin{ass}[Forward operator] \label{ass:basic} The operator $F$ is
  weakly sequentially continuous, and its domain $\domain(F)$ is
  convex and closed. Moreover, the solution $\xp \in \domain(F)$ is
  supposed to be an interior point of the domain $\domain(F)$.
\end{ass}

From \eqref{eq:m} we have that the convex penalty functional
$\norm{B(x - \bar x)}{X}^{2}$ is stabilizing. Then, taking into
account Assumption~\ref{ass:basic} we have that the Assumptions~3.11
and 3.22 in \cite{SKHK12} are all fulfilled, and the usual assertions
on existence and stability of regularized solutions
(cf.~\cite[Section~4.1.1]{SKHK12}) apply.

The unbounded self-adjoint operator~$B$ satisfying \eqref{eq:m} generates a {\it Hilbert scale} $\{X_\tau\}_{\tau \in
  \mathbb{R}}$ with $X_0=X$, $X_\tau=\domain(B^\tau)$,  and with
corresponding  norms $\|x\|_\tau:=\|B^\tau x\|_X$.

We make the following additional assumption on the structure of
non-linearity for the forward operator $F$ with respect to the Hilbert
scale generated by the operator $B$. Sufficient conditions and
examples for this non-linearity assumption will be given in the
appendix.
\begin{ass}
  [Non-linearity structure]\label{ass:nonlinearity} There is a
  number~$a>0$, and there are constants~$0< c_{a}< C_{a}<\infty$ such
  that
  \begin{equation} \label{eq:twosides} c_{a}\norm{x - \xp}{-a} \leq
    \norm{F(x) - F(\xp)}{Y} \leq C_{a}\norm{x - \xp}{-a}\; \mbox{for
      all} \; x\in \mathcal{D}.
  \end{equation}
\end{ass}

At this point we recall the concept of local well-posedness and
ill-posedness from \cite[Definition~2]{HofSch94}. A similar concept is also mentioned in \cite{LorWor13}.

\begin{definition}[local well-posedness and
  ill-posedness] \label{def:localill} The operator equation
  \eqref{eq:opeq} with forward operator $F:\domain(F) \subset X \to Y$
  is called locally well-posed at the point $\xp \in \domain(F)$ if
  there is a ball $\mathcal{B}_r(\xp)$ with center $\xp$ and radius
  $r>0$ such that, for every sequence
  $\{x_n\}_{n \in \mathbb{N}} \subset \mathcal{B}_r(\xp)\cap
  \domain(F)$, the implication
$$ \lim \limits_{n \to \infty}\|F(x_n)-F(\xp)\|_Y = 0\quad  \Longrightarrow  \quad  \lim \limits_{n \to \infty}\|x_n-\xp\|_X = 0$$
is valid, otherwise \eqref{eq:opeq} is called locally ill-posed at the
point $\xp$.
\end{definition}

If $F=A \in \mathcal{L}(X,Y)$ is a bounded linear operator, then we
note that \eqref{eq:opeq} is locally well-posed everywhere if $F$ is
injective and the range $\mathcal{R}(A)$ is a closed subset of $Y$,
otherwise the equation is locally ill-posed everywhere. For linear
$F=A$ the left inequality of \eqref{eq:twosides} yields
injectivity. For non-linear $F$ this left inequality of
\eqref{eq:twosides}, which was also exploited in the paper
\cite{Taute94}, yields local injectivity at $\xp$ and if the lower
index $a$ were zero, then we would have local well-posedness of
\eqref{eq:opeq} from \eqref{eq:twosides}. But for $a>0$, i.e.,\ when
the norm in $X_a$ is weaker than the norm in $X$, local ill-posedness
can occur and motivates the use of regularization for the stable
approximate solution of \eqref{eq:opeq}.

We turn to the notion of solution smoothness, and we shall measure
this  with respect to the operator~$B$.
\begin{ass}
  [Solution smoothness]\label{ass:smooth} There are $0<p<1$ and
  $E <\infty$ such that~$\xp\in \domain(B^{p})$ and
  \begin{equation}
    \label{eq:mpe}
    \xp - \bar x \in \mpe:= \set{x\in X_{p},\quad \norm{x}{p}:= \norm{B^{p}x}{X} \leq E}
  \end{equation}
  for the smooth reference element $\bar x\in\domain$ from the
  Tikhonov functional $T^\delta_\alpha$. However, we assume that
  $\xp \notin \domain(B)$.
\end{ass}
The focus in this study is to discuss \emph{oversmoothing} with
$\xp - \bar x \in \mpe$ for some~$0< p <1$, but $\xp \notin \domain(B)$, albeit
the Tikhonov penalty uses the operator~$B$, such that all minimizers
of~(\ref{eq:tikhonov}) are smoother than the solution element~$\xp$.

Throughout this paper, we choose the regularization parameter a
posteriori as $\adp=\adp(\delta,\yd)$ according to a specific version
of the {\it discrepancy principle}, which is introduced as follows:

\begin{DP}
  For a prescribed constant $C> 1$ chose the regularization
  $\alpha=\alpha_{DP}$ according to
  \begin{equation}
    \label{eq:dp1}
    \norm{F(\xad) - \yd}{Y} = C \delta .
  \end{equation}
\end{DP}

In the sequel we assume that $\adp$ exists for the given data
$\yd \in Y$, at least in the case of sufficiently small $\delta>0$. In
our context (cf.~\eqref{eq:tikhonov}), the condition
$$\|F(\bar x)-\yd\|_Y > C\delta$$
is sufficient for this existence in the special case of linear forward
operators $F$. Due to possibly occurring gaps the existence of
$\adp$ can fail for non-linear $F$, and additional conditions
(cf.~\cite[Condition~3.9]{AnzRam10}) are required for ensuring the
existence of $\adp$.

In the subsequent analysis we suppose that the chosen
parameter~$\alpha=\alpha(\delta,\yd)$ taken from the discrepancy
principle as $\alpha=\adp$ tends to zero as $\delta \to 0$. As the
following lemma shows, this is no additional requirement, but it is
ensured under the above assumptions.

\begin{lemma} \label{lem:onlycase} Under the assumptions stated above
  the regularization parameters $\adp = \adp(\delta,\yd)$, depending
  on $\delta$ and $\yd$, chosen according to the discrepancy principle,
  tend to zero when $\delta \to 0$ and $\yd$ obeys the noise model
  \eqref{eq:noise}.
\end{lemma}
\begin{proof} Consider a sequence~$\delta_n \to 0,$  and an
  associated sequence of noisy data $y^{\delta_n} \in Y$ with
  $\|F(\xp)-y^{\delta_n}\|_Y \le \delta_n$. Let~$\alpha_{n}$ be the obtained
  sequence of regularization parameters according to the discrepancy
  principle, hence
  with~$\|F(x_{\alpha_n}^{\delta_n})-y^{\delta_n}\|_Y=C\delta_n$.

Suppose to the contrary that there is~$\underline \alpha>0$ such
that~$\alpha_n \ge \underline \alpha$ for all $n\in \mathbb{N}$.
 Because
  $\bar{x} \in \domain(F) \cap \domain(B)$, all regularized solutions
  with regularization parameter $\underline \alpha>0$ and data $y^{\delta_n}$
  satisfy the inequality
$$\|F(x_{\underline{\alpha}}^{\delta_n})-y^{\delta_n}\|_Y^2+\underline \alpha\,\|B(x_{\underline{\alpha}}^{\delta_n} -\bar{x})\|_X^2  \le \|F(\bar{x})-y^{\delta_n}\|_Y^2,$$
which implies that there is some $K>0$ such that
$m\|x_{\underline{\alpha}}^{\delta_n} -\bar{x}\|_X \le
\|B(x_{\underline{\alpha}}^{\delta_n} -\bar{x})\|_X \le K$.
Hence, there is a weakly convergent sub-sequence
$x_{\underline{\alpha}}^{\delta_{n_k}} \rightharpoonup z$ in the
Hilbert space $X_1 = \domain(B)$.  The limit element $z$ is also the weak limit in
X and belongs to $\domain(F) \cap \domain(B)$, because the norm is
lower semi-continuous and $\domain(F)$ is weakly closed. Then the
non-decreasing property of the discrepancy norm in Tikhonov
regularization with respect to $\alpha$, the weak sequential
continuity of $F$ and the local injectivity of $F$ at $\xp$ expressed
by the left inequality of \eqref{eq:twosides} together with the
properties $z \in \domain(B)$  and $\xp \notin \domain(B)$ allow
us to bound as follows:
$$
0= \liminf\limits_{k\to \infty} C\delta_{n_k} =
\liminf\limits_{k\to \infty}
\|F(x^{\delta_{n_k}}_{\alpha_{n_k}})-y^{\delta_{n_k}}\|_Y $$ $$ \ge
\liminf\limits_{k\to \infty}
\|F(x^{\delta_{n_k}}_{\underline{\alpha}})-y^{\delta_{n_k}}\|_Y  =
\|F(z)-F(\xp)\|_Y>0.
$$
This is a contradiction, and completes the proof of the lemma.
\end{proof}

\begin{remark}\label{rem:penalty}
  It can be seen from the reasoning in the above proof, that under the
  present assumptions the penalty~$\|B(\xad -\bar{x})\|_X^{2}$ must
  tend to infinity  for the
  corresponding~$\alpha = \adp$, provided that~$\delta\to 0$.
\end{remark}

\pagebreak

The main result is the following.
\begin{thm}
Consider the operator equation~$F(x) = y$ in Hilbert spaces, with forward operator $F\colon
\domain(F)\subset X\to Y$  which obeys Assumptions~\ref{ass:basic}
and~\ref{ass:nonlinearity}. Let~$\xad$ be the minimizer of the Tikhonov minimization problem
(\ref{eq:tikhonov}), where the regularization parameter~$\alpha=\adp$ is
chosen according to the discrepancy principle.
 If
  $\xp \in \domain(F)$ has smoothness as in Assumption~\ref{ass:smooth} for some
  $0< p <1$ (oversmoothing penalty), then $\xad$ yields the convergence rate
  $$
\norm{\xad - \xp}{} = \bigo(\delta^{p/(a+p)})\quad \mbox{as} \quad \delta \to 0.
$$
\end{thm}
We postpone a discussion of this result to
Section~\ref{sec:discussion}, and we first  prove the theorem in the next section.

\section{Proof of the main result}
\label{sec:proof-main}

In various places we shall use interpolation inequalities in the Hilbert scale, which are generated by
the operator~$B$, see e.g.~\cite[Prop.~8.19]{EHN96}. These inequalities ensure,
for given numbers~$-a < r \leq s$, that we obtain
\begin{equation}
  \label{eq:interpol}
  \norm{x}{r} \leq
  \norm{x}{s}^{\frac{r+a}{s+a}}\norm{x}{-a}^{\frac{s-r}{s+a}},\quad
  x\in X_{s}= \domain(B^{s}).
\end{equation}
\medskip

In a first step we bound the error in the weak norm exploiting the
lower bound of the inequality chain \eqref{eq:twosides}.
\begin{lemma}\label{lem:-a-norm}
  Suppose that the Assumptions~\ref{ass:basic} and
  \ref{ass:nonlinearity} hold. If the regularization
  parameter~$\alpha=\adp$ is chosen according to the discrepancy
  principle then
  \begin{equation}
    \label{eq:-a-norm}
    \norm{\xad - \xp}{-a} \leq \left(\frac{C+1}{c_{a}}\right)\delta.
  \end{equation}
\end{lemma}
We mention the following consequence.
\begin{pro}\label{pro:interpol-tilde}
Let~$\alpha = \adp$ be chosen according to the discrepancy principle.
  Suppose that~$\norm{B^{p}(\xad - \bar x)}{X}\leq \tilde E$ holds
  true for some constant $\tilde E>0$. Then we have under
  Assumptions~\ref{ass:basic}--\ref{ass:smooth} that
  $$
  \norm{\xad - \xp}{X } \leq (E + \tilde E)^{a/(a+p)}
  \left(\frac{C+1}{c_{a}}\, \delta\right)^{p/(a+p)}.
$$
\end{pro}
\begin{proof}
  From the bound on~$\norm{B^{p}(\xad - \bar x)}{X}$ and
  Assumption~\ref{ass:smooth} we find
  that~$\norm{\xad - \xp}{p}\leq E + \tilde E$.  Now we use the
  interpolation inequality~(\ref{eq:interpol}) with~$r=0,\ s=p$, yielding
$$
\norm{\xad - \xp}{X}\leq \norm{\xad - \xp}{p}^{a/(a+p)}\norm{\xad -
  \xp}{-a} ^{p/(a+p)}.
$$
With lemma~\ref{lem:-a-norm} this completes the proof.
\end{proof}
We thus make the important observation that  it is enough to show
that~$\norm{B^{p}(\xad - \bar x)}{X}\leq \tilde E$ holds
true  as~$\delta\to 0$ and hence~$\alpha =\adp\to 0$.

\subsection{Auxiliary linear problem and auxiliary element}
\label{sec:aux-problem}

For the subsequent analysis the following linear Tikhonov
regularization will be important.  Consider the Tikhonov functional
\begin{equation}
  \label{eq:Tikh-aux}
  \Jaa (x) := \norm{x - \xp}{-a}^{2} + \alpha \norm{B (x - \bar x)}{X}^{2},\quad
  x\in X.
\end{equation}
\begin{pro}\label{pro:aux-problem}
Suppose that~$\xp \in \mpe$ for some $0<p<1$.
  Given $\alpha>0$ let~$\xa$ be the minimizer of~$\Jaa$.
Then
  \begin{align}
    \norm{\xa - \xp}{X} &\leq E \alpha^{p/(2(a+1))}, \label{it:xa-xp}\\
 \norm{B^{-a}(\xa - \xp)}{X} &\leq E \alpha^{(a+p)/(2(a+1))},\label{it:xa-xp-a}\\
\norm{B(\xa - \bar x)}{X} & \leq E \alpha^{(p-1)/(2(a+1))}, \label{it:B-xa}
    \intertext{and}
    \Jaa(\xa) & \leq 2E^{2} \alpha^{(a+p)/(a+1)}.\label{it:Jaa}
  \end{align}
Moreover, we have that
\begin{equation}
  \label{eq:xa-p-norm}
  \norm{\xa - \bar x}{p} \leq E,\  \text{and}\   \norm{\xa - \xp}{p} \leq E.
\end{equation}
\end{pro}
\begin{proof}
  The minimizer~$\xa$ has the explicit form
  \begin{equation*}
    \xa = \bar x + \lr{B^{-2a} + \alpha B^{2}}^{-1} B^{-2a}(\xp - \bar
    x).
  \end{equation*}
  Introducing the bounded self-adjoint operator~$H:= B^{-2(a+1)}$ we
  can rewrite this as
  $ \xa = \bar x + \lr{\alpha I + H}^{-1} H (\xp - \bar x).  $ Hence,
  with $$\ra(t) = \alpha/(t+\alpha),\;\; t>0\,,$$ we find
$$
\norm{\xa - \xp}{X} = \norm{\ra(H)( \xp - \bar x)}{X}.
$$
Now we use that~$\xp\in\mpe$. First, this yields
$$
\norm{\xa - \xp}{X} \leq E \norm{\ra(H)H^{p/(2(a+1))}}{X \to X}\leq E
\alpha^{p/(2(a+1))},
 $$
 since $\ra(t) t^q \leq \alpha^q$ whenever~$0 < q \leq 1, \ t>0,$ due
 to the qualification of linear Tikhonov regularization.  Similarly, because of~$0<p < 1< a+2$, we obtain the estimate
 \begin{align*}
   \norm{B^{-a}(\xa - \xp)}{X} & \leq E \norm{B^{-(a+p)}\ra(H) }{X\to X} = E
                          \norm{\ra(H) H^{(a+p)/(2(a+1))} }{X\to X}  \\
                        & \leq E \alpha^{(a+p)/(2(a+1))},
 \end{align*}
 by using the qualification of Tikhonov regularization, again.

 In the same manner we bound~$\norm{B(\xa - \bar x)}{X}$ as
$$
\norm{B(\xa - \bar x)}{X} \leq E \alpha^{(p-1)/(2(a+1))},
$$
and the simple calculation
$$
\alpha\norm{B(\xa - \bar x)}{X}^{2} \leq E^{2} \alpha^{\lr{1 +
    (p-1)/(a+1)}} = E^{2} \alpha^{(a+ p)/(a+1)}
$$
completes the bound for the Tikhonov functional~$\Jaa$.
For the final assertion we start with bounding~$\norm{\xa - \bar x} p$
as
$$
\norm{\xa - \bar x}{p} = \norm{B^{p} \lr{\alpha I + H}^{-1} H(\xp- \bar x)}{X} \leq
\norm{B^{p}(\xp- \bar x)}{X\to X} \leq E,
$$
and similarly
$$
\norm{\xa - \xp}{p} = \norm{B^{p}\ra(H)(\xp- \bar x)}{X} \leq
\norm{B^{p}(\xp- \bar x)}{X} \leq E.
$$
The proof is complete.
\end{proof}
\begin{remark}\label{rem:xa-near}
  From the first
  bound~$\norm{\xa - \xp}{X} \leq E \alpha^{p/(2(a+1))}$ we see that,
  for $\alpha>0$ small enough, the element~$\xa$ from the auxiliary
  linear problem will belong to the domain~$\domain(F)$, because $\xp$
  is an inner point of $\domain(F)$.
\end{remark}

Since the approximations~$\xad$ are smoother than the solution
element~$\xp$, with respect to the operator $B$ generating the Hilbert scale, we shall employ some intermediate element of higher
smoothness. The principal idea of using such auxiliary elements was
already mentioned in~\cite{Plato90} (see also
\cite[Theorem~10.7]{EHN96} and \cite{Jin-Hou1999,LPR07} with more connection to
our approach). To this end we shall employ in this context an element~$\xa=\xaux$ from
the family $\xa\;(\alpha >0)$ of regularized solutions minimizing the
functional \eqref{eq:Tikh-aux}, with properties listed in
Proposition~\ref{pro:aux-problem}, but for the parameter~$\alpha = \aux(\delta)$
depending on the noise level $\delta$ according to a specific discrepancy principle given below in formula \eqref{eq:choice-alpha-xa}.
\begin{definition}
  [Auxiliary element]\label{def:auxiliary}
Consider  the family $\xa\;(\alpha >0)$ of regularized solutions minimizing the functional \eqref{eq:Tikh-aux} and
let $\delta>0$ be sufficiently small such that there is some $\alpha=\aux(\delta)$ satisfying the equation
\begin{equation}
  \label{eq:choice-alpha-xa}
  \norm{F(\xa) - F(\xp)}{Y} = (C-1) \delta,
\end{equation}
with constant $C>1$ as used in \eqref{eq:dp1}.
Then we call the corresponding element $\xaux := x_{\aux}(\delta)$ \emph{auxiliary element}.
\end{definition}

The auxiliary element $\xaux \in \domain(F)$ exists for sufficiently small $\delta>0$, because $\xp$ is an inner point of $\domain(F)$ and we have for $\alpha \to 0$ that \linebreak
$\|\xa-\xp\|_X \to 0$ and thus, as a consequence of the right inequality in \eqref{eq:twosides}, $\|F(\xa)-F(\xp)\|_Y \to 0$.
Moreover, the element $\xaux$ is uniquely determined and  has the following optimality properties.
\begin{pro}\label{pro:aux-element}
  The auxiliary element obeys
  \begin{align}\norm{\xaux(\delta) - \xp}{X} &
\leq E^{a/(a+p)} \lr{\frac{C}{c_{a}}\delta}^{p/(a+p)},\\
\norm{B(\xaux(\delta) - \bar x)}{X}
& \leq E \lr{\frac{C_{a}
  E}{C-1}\delta}^{\frac{p-1}{a+p}}, \label{it:B-xaux-barx}
  \intertext{and}
\norm{\xaux(\delta) - \xp}{-a} & \leq \frac{C-1}{c_{a}}\delta.\label{it_-a-bound}
  \end{align}
\end{pro}
\begin{proof}
  The first assertion is an easy application of interpolation, similar to
  Proposition~\ref{pro:interpol-tilde}, using the bound from
  Lemma~\ref{lem:-a-norm} with $C$ instead of $C-1$, and the
  bound~(\ref{eq:xa-p-norm}) from Proposition~\ref{pro:aux-problem}.

For the second assertion we exploit the fact that, at the parameter~$\aux(\delta)$ and
by favour of the estimate~(\ref{it:xa-xp-a}), we have
$$
(C-1) \delta = \norm{F(\xaux) - F(\xp)}{Y} \leq C_{a} \norm{\xaux -
  \xp}{-a} \leq C_{a}E \aux^{\frac{a+p}{2(a+1)}}.
$$
This provides us with the lower bound
\begin{equation}
  \label{eq:aux-lb}
\aux(\delta) \geq \lr{\frac{C-1}{C_{a}E}\delta}^{\frac{2(a+1)}{a+p}}.
\end{equation}
Plugging this bound into the estimate~(\ref{it:B-xa}) we find that
$$
\norm{B(\xaux(\delta) - \bar x)}{X} \leq E \lr{\frac{C-1}{C_{a}E}\delta}^{\frac{p-1}{a+p}},
$$
where we used that~$0< p <1$. For the last assertion we use Definition
of~$\xaux$ and the (non-linearity) Assumption~\ref{ass:nonlinearity} to derive
$$
c_{a} \norm{\xaux - \xp}{-a} \leq \norm{F(\xaux) - F(\xp)}{Y} = (C-1)\delta,
$$
from which the proof can easily be completed.
\end{proof}
\begin{remark}
  It will be seen from Lemma~\ref{lem:bounding}, below, that a similar
  upper bound can be shown for the
  growth of the penalty~$\norm{B(\xad - \bar x)}{X}$.
\end{remark}

\subsection{Bounding the term $\norm{B^p(\xad - \bar x)}{X}$}
\label{sec:bounding-Bpxad}

Our goal is to apply Proposition~\ref{pro:interpol-tilde}, and we thus
aim at proving that
$\norm{B^{p}(\xad- \bar x)}{X}\leq \tilde E$.
Since we already know
that the auxiliary element~$\xaux$ obeys~$\norm{\xaux -
  \bar x}{p}\leq E$, it  is sufficient to obtain~$\norm{\xaux - \xad}{p}\leq
\hat E$ such that $\tilde E=E+\hat E$.

\begin{lemma}\label{lem:bounding}
 Let $C>1$ be the constant used in \eqref{eq:dp1} as well as in \eqref{eq:choice-alpha-xa}  and let parameters $\alpha=\adp$
 be chosen according to the discrepancy principle.
Then there is a constant~$\hat E$ depending on $a,p,C,c_a,C_a$ and $E$, but
 not on  the noise level~$\delta$ for which
$$
\norm{B^{p}(\xaux - \xad)}{X}\leq \hat E.
$$
\end{lemma}
\begin{proof}
  The proof will be done in two steps.

  Inserting~$\xaux$ into the Tikhonov functional~(\ref{eq:tikhonov}) we find, by using that
  both~$\alpha$ were chosen according to the discrepancy
  principle and $\xaux$ was chosen as
  in~(\ref{eq:choice-alpha-xa}), that
  \begin{align*}
\lr{C \delta}^{2} +  \alpha \norm{B(\xad - \bar x)}{X}^{2} & \leq
\norm{F(\xad) - \yd}{Y}^{2} + \alpha \norm{B(\xad - \bar x)}{X}^{2}\\
& \leq \norm{F(\xaux) - \yd}{Y}^{2} + \alpha \norm{B(\xaux - \bar
                                                                     x)}{X}^{2} \\
& \leq \lr{(C-1) \delta + \delta}^{2}    + \alpha \norm{B(\xaux - \bar x)}{X}^{2}\\
& = \lr{C \delta}^{2} +  \alpha \norm{B(\xaux - \bar  x)}{X}^{2},
  \end{align*}
  such that~$\norm{B(\xad - \bar x)}{X}\leq \norm{B(\xaux - \bar
    x)}{X}$.
Using the triangle inequality we find
\begin{equation} \label{eq:addnew}
\norm{B(\xaux - \xad)}{X} \leq \norm{B(\xaux - \bar x)}{X}
+\norm{B (\xad-\bar x)}{X} \leq 2  \norm{B(\xaux - \bar
  x)}{X}.
\end{equation}
Now we use the interpolation
  inequality~(\ref{eq:interpol}). Since~$-a <  p < 1$ we have the estimate
\begin{equation}
  \label{eq:new}
\norm{\xaux - \xad}{p}\leq \norm{\xaux -
  \xad}{1}^{\frac{a+p}{a+1}}\norm{\xaux - \xad}{-a}^{\frac{1-p}{a+1}}.
\end{equation}
For the first factor in~(\ref{eq:new}) we apply inequality (\ref{eq:addnew})
and the corresponding
bound~(\ref{it:B-xaux-barx}) from
Proposition~\ref{pro:aux-element}. From the bound~(\ref{it_-a-bound}) in
Proposition~\ref{pro:aux-element} and Lemma~\ref{lem:-a-norm} we find
that
$$
\norm{\xaux - \xad}{-a} \leq \norm{\xaux - \xp}{-a} + \norm{\xad -
  \xp}{-a} \leq
\left(\frac{2C}{c_{a}}\right)\delta,
$$
which allows bounding the second factor in~(\ref{eq:new}).
Overall this yields
$$
\norm{\xaux - \xad}{p}\leq \lr{2E \lr{\frac{C_{a}
  E}{C-1}\delta}^{\frac{p-1}{a+p}}}^{\frac{a+p}{a+1}}
\lr{\left(\frac{2C}{c_{a}}\right)\delta}^{\frac{1-p}{a+1}} =: \hat E,
$$
 where the constant~$\hat E$ depends on $a,p,C,c_a,C_a$ and $E$, but
 not on  the noise level~$\delta$.
Since we have from the estimate (\ref{eq:xa-p-norm})  in
Proposition~\ref{pro:aux-element} (which holds for
all~$\alpha>0$, and hence also for~$\aux$) that
$\norm{\xaux - \bar x}{p}\leq E$, we conclude
that~$\norm{B^{p}(\xad - \bar x)}{}\leq \hat E + E$.
\end{proof}

 This allows us to
apply Proposition~\ref{pro:interpol-tilde}, and it yields
the order optimality of the reconstruction~$\xad$ as stated in the theorem.

\section{Discussion}
\label{sec:discussion}

First, we mention that in the limiting case~$p=1$ the order optimal
convergence rate as established in the theorem is well-known also for the discrepancy principle
(cf.~Theorem~3.1 in~\cite{Taute94}). In this situation, where the penalty is not oversmoothing,
for the proof of the rate result only the left inequality of \eqref{eq:twosides} is required, and
no auxiliary elements are needed. For other regularization
  methods, which do not use a penalty, the rates as established in the
  theorem are known. This is the case for Landweber iteration, see
  e.g.~\cite{MR1754723}, and for Newton-type methods, see~\cite{MR1741236}.

The theorem also applies for linear operators, and in
this case it extends the original result
from~\cite{Natterer84} to the a posteriori parameter choice by the
discrepancy principle.

We add the following observation.
The lower bound from~(\ref{eq:aux-lb})  was
    derived for the parameter choice according to the discrepancy
    principle. We notice that this bound coincides up to different constants with the
  a priori parameter choice $\alpha(\delta)=\alpha_{apriori}(\delta)$
  established in~\cite{Natterer84}  and obeys the limit condition
  \begin{equation}\label{eq:infi}
  \frac{\delta^2}{\alpha(\delta)} \to \infty \quad \mbox{as} \quad \delta \to 0.
  \end{equation}

In the conventional case, i.e.,\  when $\xp \in
\domain(B)$,
we always have (also for the discrepancy principe, cf.~\cite[Proposition~8]{AnzHofMat14})
that $$\frac{\delta^2}{\alpha(\delta)} \to 0 \quad \mbox{as} \quad
\delta \to 0.$$
Together with the fact that~$\alpha(\delta) \to 0$
this provides the basis of the convergence theory for regularized
solutions when $\xp \in \domain(B)$
(cf.~\cite[Section~4.1.2]{SKHK12}).
Within regularization with oversmoothing penalty, i.e., when~$\xp
\notin \domain(B)$,  this cannot be
expected and even the limit condition (\ref{eq:infi}) seems to be possible in general.

We also highlight that the optimal rate as described here shows a
delay of the saturation phenomenon, known for Tikhonov
regularization. To exhibit this let us stick to the linear case,
and let us take~$B:=\lr{A^{\ast}A}^{-1}$, hence we have that~$\xp \in\domain(B)$ exactly
if~$\xp\in\mathcal{R}(A^*A)$. The link condition from
Assumption~\ref{ass:nonlinearity} then holds for~$a=1/2$, and the
maximal rate (for $p=1$, i.e., no oversmoothing), as established in
the present study is~$\delta^{1/(1 +   1/2)} = \delta^{2/3}$, even
under the discrepancy principle. Without the smoothness promoting
penalty, i.e.,\ when the penalty is~$\norm{x - \bar x}{X}^{2}$, then,
under the discrepancy principle, the best possible rate
is~$\delta^{1/2}$, which is attained for maximal smoothness~$p=1/2$,
i.e.,\ when~$\xp \in \mathcal{R}(A^{\ast})$.

\section*{Acknowledgment}
The research was financially supported by Deutsche
Forschungsgemeinschaft (DFG-grant HO 1454/10-1).

\appendix
\section*{Appendix} \label{sec:appendix}

We collect assertions and examples which highlight that the
non-linearity condition~\eqref{eq:twosides} of
Assumption~\ref{ass:nonlinearity} is valid in specific cases.  For a
more detailed discussion of the interplay between the occurring norms
$\norm{F(x)-F(\xp)-F^\prime(\xp)(x -\xp)}{Y},$ $\norm{F(x) - F(\xp)}{Y}$ and also \linebreak $\norm{F^\prime(\xp)(x - \xp)}{Y}$ in the case of
ill-posed problems we refer to \cite{HofSch94}.

We start from the following result, which is easily verified, and we
omit the proof.
\begin{pro}
  \label{pro:appendix}
  Suppose that the operator $F$ is continuously Fr\'echet
  differentiable in a ball $\mathcal{B}_r(\xp) \subset \domain(F)$
  with radius $r>0$ and let denote the Fr\'echet derivatives by
  $F^\prime(x),\;x \in \mathcal{B}_r(\xp)$. If there are
  constants~$0 < g_{a} \leq G_{a}<\infty$
  and~$0 < k_{a}< K_{a}<\infty$ such that
  \begin{align}
    g_{a}\|B^{-a}h\|_X & \leq
                         \|F^\prime(\xp)h\|_Y \leq G_{a}\|B^{-a}h\|_X\quad \mbox{for all} \;\; h\in X,\label{eq:Gateaux}\intertext{and}
                         k_{a}\norm{F^\prime(\xp)(x -
                         \xp)}{Y} &\leq \norm{F(x) - F(\xp)}{Y} \leq
                                    K_{a}\norm{F^\prime(\xp)(x - \xp)}{Y},\label{eq:Flowup}
  \end{align}
  for all $x\in \mathcal{B}_r(\xp)$, then
  Assumption~\ref{ass:nonlinearity} holds with
  constants~$c_{a}= g_{a} k_{a}$ and~$C_{a}= G_{a}K_{a}$.
\end{pro}
\begin{remark}
  The inequality chain \eqref{eq:Gateaux} restricts the class of
  possible operators $B$ for generating the Hilbert scale under
  consideration. The positive constant $a$ can be interpreted as {\it
    degree of ill-posedness} for the linearized problem
  $F^\prime(\xp)\,x=y$.  We mention that the condition
  \eqref{eq:Gateaux} has already been used in \cite[Assumption
  2.1~(iv)]{Neubau92}).  Since $B$ is assumed to be unbounded with
  lower bound $m>0$ of the spectrum $\sigma(B)$, the operator $B^{-a}$
  is bounded and injective with
  $\|B^{-a}\|_{X \to X} \le \frac{1}{m^a}$ and $0 \in \sigma(B^{-a})$.
  Consequently, $F^\prime(\xp)$ in \eqref{eq:Gateaux} must be
  injective with non-closed range and
  $0 \in \sigma((F^\prime(\xp)^*F^\prime(\xp))^{1/2})$.
\end{remark}

We close this section with two corollaries of Proposition~\ref{pro:appendix} and corresponding examples,
where the non-linearity condition from~(\ref{eq:Flowup}) can be proven in
specific situations.
\begin{corollary}
  Suppose that the operator $F$ is continuously Fr\'echet
  differentiable in a ball $\mathcal{B}_r(\xp) \subset \domain(F)$
  with radius $r>0$. If, for a constant $0 \le \eta <1$, the
  non-linearity condition of tangential cone type
  (cf.~\cite[Section~11.1]{EHN96})
  \begin{equation} \label{eq:Fc} \norm{F(x)-F(\xp)-F^\prime(\xp)(x -
      \xp)}{Y} \leq \eta \norm{F(x) - F(\xp)}{Y}
  \end{equation}
  is satisfied for all $x\in \mathcal{B}_r(\xp)$ then
  \eqref{eq:Flowup} holds with $k_a=1/(1+\eta)$ and
  $K_a=1/(1-\eta)$. Hence~(\ref{eq:Gateaux}) and~(\ref{eq:Fc}) imply
  that Assumption~\ref{ass:nonlinearity} is valid.
\end{corollary}

\begin{xmpl}[Exponential growth model]\label{ex:exp-growth} The
  following example for an  exponential growth model was originally
introduced by Groetsch in
  \cite[Section~3.1]{Groe93}. Here we aim at identifying  the time dependent growth rate
  $x(t)\;(0 \le t \le T)$ from observations of the size
  $y(t)\;(0 \le t \le T)$ of a population with $y(0)=y_0>0$ such that
  the initial value problem
$$ y^\prime(t) = x(t)\,y(t)\quad(0 \le t \le T),\qquad y(0)=y_0\,,$$
is satisfied. For $X=Y=L^2(0,T)$ the forward operator attains the form
\begin{equation*}
[F(x)](t)=y_0\,\exp\left(\int_0^t x(\tau)d\tau \right) \quad (0 \le t \le T).
\end{equation*}
We note that the corresponding operator equation $F(x) = y$
is locally ill-posed everywhere in $X$.  The operator $F$ is Fr\'echet
differentiable on the whole Hilbert space $L^2(0,1)$ and possesses the
Fr\'echet derivative
$$[F^\prime (\xp) h](t)=[F(\xp)](t) \int_0^t h(\tau)d\tau  \quad (0 \le t \le T), \quad h \in L^2(0,T), $$
continuously mapping in $L^2(0,T)$. One easily verifies that
\begin{equation} \label{eq:nldegree}
  \|F(x)-F(\xp)-F^\prime(\xp)(x-\xp)\|_Y \le
  K\,\|F(x)-F(\xp)\|_Y\|x-\xp\|_X
\end{equation}
holds with some constant $K>0$ for all $x \in X$. Hence, for
$\domain(F)=\overline{\mathcal{B}_r(\xp)}$ (closed ball around $\xp$
with sufficiently small radius $r>0$) the condition \eqref{eq:Fc} is
satisfied with $0<\eta<1$.

Looking at~(\ref{eq:nldegree}) it is reasonable to consider the integration operator
$[Jh](t):=\int_0^t h(\tau)d\tau \;(0 \le t \le T)$ mapping in
$L^2(0,T)$, and the related Hilbert scale $\{X_\tau\}_{\tau \in
  \mathbb{R}}$ generated by the operator $B=(J^*J)^{-1/2}$.
With respect to this scale, and since
$0<\underline c \le [F(\xp)](t) \le \overline c \le \infty\;(0 \le t
\le T)$ for the multiplier function in $F^\prime (\xp)$, we find
that~\eqref{eq:twosides} holds with $a=1$.

The range $\mathcal{R}((J^*J)^{-p/2})$ is well-known as the fractional
Sobolev space $H^p(0,T)$ for $0<p<1/2$. It is  a subspace of
$H^p(0,T)$ for $1/2 \le p \le 1$ (cf.~\cite[Thm.~2.1]{GorLuYam15}). The source condition $\xp \in \domain(B^p)$ for $0<p<1$ in combination with $\xp \notin \domain(B)$ implies
that the function $\xp$ belongs to the Sobolev space $H^p(0,T)$, but not to $H^1(0,T)$.

\end{xmpl}

\begin{corollary}
  Suppose that the operator $F$ is continuously Fr\'echet
  differentiable in a ball $\mathcal{B}_r(\xp) \subset \domain(F)$
  with radius $r>0$. If there is a family~$M(x):X \to X$ of bounded
  linear operators which obey $\|M(x)-I\|_{X \to X} \le c<1$ such
  that
  \begin{equation} \label{eq:Kalten} F^\prime(x)= M(x)\,F^\prime(\xp)
    \quad \mbox{for all} \quad x \in \mathcal{B}_r(\xp),
  \end{equation}
  then the estimates from~(\ref{eq:Flowup}) hold with
  constants~$k_a=1-c$ and $K_a=1+c$. Thus~(\ref{eq:Gateaux})
  and ~(\ref{eq:Kalten}) imply the validity of
  Assumption~\ref{ass:nonlinearity}.
\end{corollary}
\begin{proof}
By mean value theorem in integral form we have for $x \in \mathcal{B}_r(\xp)$
$$\|F(x)-F(\xp)-F^\prime(\xp)(x-\xp)\|_Y =\|\int_0^1 [F^\prime(\xp+t(x-\xp))-F^\prime(\xp)](x-\xp)dt\|_Y $$
$$\le\|\int_0^1 [M(\xp+t(x-\xp))-I]\,F^\prime(\xp)(x-\xp)\,dt\|_Y \le c \,\|F^\prime(\xp)(x-\xp)\|_Y$$
and hence by triangle inequality
$$ (1-c)\,\|F^\prime(\xp)(x-\xp)\|_Y \le \|F(x)-F(\xp)\|_Y \le (1+c)\,\|F^\prime(\xp)(x-\xp)\|_Y\,.$$
\end{proof}

The condition (\ref{eq:Kalten}) and variants of requirements on $M(x)$ have been discussed in \cite{HNS95} (see also in \cite{Kalten08}), and we
recall Example 4.1 from  \cite{HNS95}:

\begin{xmpl}[Non-linear Hammerstein integral operator]
We consider 
the spaces $X:=H^1[0,1]$ and $Y:=L^2(0,1)$ and in this context the non-linear integral operator $F:X \to Y$, which is defined as
\begin{equation*}
[F(x)](t)=\int \limits_0^t \phi(x(s))ds\,, \quad 0 \le t \le 1\,.
\end{equation*}
Then, for a sufficiently smooth real function $\phi$ and for a sufficiently small radius $r>0$, the condition (\ref{eq:Kalten}) is satisfied with
\linebreak $\|M(x)-I\|_{X \to X} \le c<1$ for arbitrary $\xp \in X$
whenever the derivative $\phi^\prime$ of the kernel function $\phi$ is uniformly bounded below by a positive constants for all arguments under consideration.
\end{xmpl}

\bibliographystyle{plain} \bibliography{revise_2nd_oversmoothing}
\end{document}